\documentclass[12pt, reqno]{amsart}
\usepackage{amsmath, amsthm, amscd, amsfonts, amssymb, mathrsfs, graphicx, color}
\usepackage[bookmarksnumbered, colorlinks, plainpages]{hyperref}
\input{mathrsfs.sty}

\newtheorem{theorem}{Theorem}[section]
\newtheorem{lemma}[theorem]{Lemma}
\newtheorem{proposition}[theorem]{Proposition}
\newtheorem{corollary}[theorem]{Corollary}
\theoremstyle{definition}
\newtheorem{definition}[theorem]{Definition}
\newtheorem{example}[theorem]{Example}

\theoremstyle{remark}
\newtheorem{remark}[theorem]{Remark}
\numberwithin{equation}{section}

\def\a{\alpha}
\def\b{\beta}

\def\l{\lambda}

\def\|{\parallel}
\def\<{\left \langle}
\def\>{\right \rangle}

\def\s{\sharp}

\begin{document}
\setcounter{page}{1}

\title[Positive Multilinear Mappings]{ Developed Matrix inequalities via Positive Multilinear Mappings}
\author[M. Dehghani, M. Kian and Y. Seo]{ Mahdi Dehghani, Mohsen  Kian and Yuki Seo }

\address{Mehdi Dehghani: Department of Pure Mathematics, Faculty of Mathematical Sciences, University of Kashan, Kashan, Iran}
\email{\textcolor[rgb]{0.00,0.00,0.84}{e.g.mahdi@gmail.com }}

\address{Mohsen Kian:\ \
 Department of Mathematics, Faculty of Basic Sciences, University of Bojnord, P. O. Box
1339, Bojnord 94531, Iran}
\email{\textcolor[rgb]{0.00,0.00,0.84}{kian@ub.ac.ir \ and \ kian@member.ams.org}}

\address{Yuki Seo: Department of Mathematics Education, Osaka Kyoiku University, Asahigaoka Kashiwara Osaka
582-8582, Japan.}
\email{\textcolor[rgb]{0.00,0.00,0.84}{ yukis@cc.osaka-kyoiku.ac.jp }}


\begin{abstract}
Utilizing the notion of  positive multilinear mappings, we give some matrix inequalities. In particular,   Choi--Davis--Jensen and Kantorovich  type  inequalities including positive multilinear mappings are presented.
\end{abstract}
\maketitle
\section{Introduction}
Throughout the paper assume that  $\mathcal{M}_p(\mathbb{C}):=\mathcal{M}_p$ is the $C^*$-algebra  of all $p\times p$ complex matrices with the identity matrix $I$. A matrix $A\in \mathcal{M}_p$ is   positive (denoted by $A\geq0$) if it is Hermitian and all its eigenvalues are nonnegative. If in addition $A$ is invertible, then it is called strictly positive (denoted by $A>0$).  A linear mapping  $\Phi:\mathcal{M}_n \rightarrow\mathcal{M}_p$ is called positive if $\Phi$ preserves the positivity and $\Phi$ is called unital if $\Phi(I)=I$. The notion of positive linear mappings on $C^*$-algebras  is a well-known tool to study matrix inequalities which  have been studied by many mathematicians. If $\Phi:\mathcal{M}_n \rightarrow\mathcal{M}_p$ is a unital positive linear mapping, then the famous Kadison's  inequality  states that  $\Phi(A)^2 \leq \Phi(A^2)$ for every Hermitian matrix $A$. Moreover, Choi's inequality   asserts  that $\Phi(A)^{-1}\leq \Phi(A^{-1})$ for every strictly positive matrix $A$, see e.g. \cite{Bh, Choi, Davis}. A continuous real function $f:J\to\mathbb{R}$ is said to be matrix convex if  $f\left(\frac{A+ B}{2}\right)\leq \frac{f(A)+ f(B)}{2}$ for all Hermitian matrices $A,B$ with eigenvalues in $J$.  Positive linear mappings have been used to characterize matrix convex and matrix monotone functions. It is well-known that a continuous real function $f:J\to\mathbb{R}$ is matrix convex if and only if $f(\Phi(A))\leq \Phi(f(A))$ for every unital positive linear mapping $\Phi$ and every Hermitian matrix $A$ with spectrum in $J$. The later inequality is known as the Choi-Davis-Jensen inequality, see \cite{Choi,Davis, FMPS}. For more information about matrix  inequalities concerning positive linear mappings the reader is referred to \cite{AuH,BhSh,BR,Furuta,KSA, KM,Lin,MPP,N} and references therein.

It is known that if $A=(a_{ij})$ and $B=(b_{ij})$ are positive matrices, then so is their Hadamard product $A\circ B$, which is defined by the entrywise product as $A\circ B=(a_{ij}b_{ij})$. The same is true for their tensor product $A\otimes B$, see e.g. \cite{HJ}.  Moreover, the mapping $(A,B)\to A\otimes B$ is also linear in each of its variables. So if we define  $\Phi:\mathcal{M}_q^2\to \mathcal{M}_p$ by $\Phi(A,B)=A\otimes B$, then $\Phi$ is multilinear  and positive in the sense that $\Phi(A,B)$ is positive whenever $A,B$ are positive. However, the Choi--Davis--Jensen type  inequality $f(\Phi(A,B))\leq\Phi(f(A),f(B))$ does not  hold in general for a unital positive multilinear mapping $\Phi$ and   matrix convex functions $f$. This can be a motivation to study matrix inequalities via positive multilinear mappings.

 The main aim of the present work is to inquire  matrix  inequalities using  positive multilinear mappings. Some of these inequalities would be generalization of inequalities for the Hadamard product and the tensor product of matrices. In Section 2, we give basic facts and examples for positive multilinear mappings. In Section 3, Choi--Davis--Jensen type  inequalities are presented for positive multilinear mappings and some of its applications is given. In Section 4, we obtain some convex theorem for functions  concerning positive multilinear mappings and then we give some Kantorovich type inequalities. Section 5 is devoted to some reverses of the Choi--Davis--Jensen inequality. In Section 6, we consider the Karcher mean and present multilinear versions of some known results.

\section{Preliminaries}
We start by definition of a positive multilinear mapping.

\begin{definition}
  A mapping $\Phi:\mathcal{M}_q^k\to \mathcal{M}_p$ is said to be multilinear if it is linear in each of its variable. It is called positive if $\Phi(A_1,\cdots, A_k)\geq0$, whenever $A_i\geq0$ \ $(i=1,\cdots, k)$. It is called strictly positive if $\Phi(A_1,\cdots, A_k)>0$ whenever $A_i>0$ \ $(i=1,\cdots, k)$. If $\Phi(I,\cdots, I)=I$, then $\Phi$ is called unital.
\end{definition}
\begin{example}
The tensor product of every two positive matrices is positive again. This ensures that the mapping $\Phi:\mathcal{M}_q^k\to \mathcal{M}_{q^k}$ defined by
\[
\Phi(A_1,\cdots,A_k)=A_1\otimes\cdots \otimes A_k
\]
is positive. Moreover, it is multilinear and unital. If $\Psi:\mathcal{M}_{q^k}\to \mathcal{M}_{p}$ is a unital positive linear mapping, then the map defined by
\[
\Phi(A_1,\ldots,A_k) = \Psi(A_1\otimes\cdots \otimes A_k)
\]
is a unital positive multilinear mapping. In particular,  the Hadamard product
 $$\Phi(A_1,\cdots,A_k)=A_1\circ\cdots\circ A_k.$$
 is positive and multilinear.
\end{example}
\begin{example}
 Assume that $X_i\in\mathcal{M}_q$ \ $(i=1,\cdots, k)$ and $\sum_{i=1}^{k}X_i^*X_i=I$.  The mapping $ \Phi:\mathcal{M}_q^k\to \mathcal{M}_p$ defined by $\Phi(A_1,\cdots,A_k)=\sum_{i=1}^{k}X_i^*A_iX_i$ is positive and unital. However, it is not multilinear.
\end{example}
\begin{example}
The mappings $\Phi:\mathcal{M}_q^k\to \mathcal{M}_p$ defined by
  \begin{align}
    \Phi(A_1,\cdots, A_k):=\mathrm{Tr}(A_1)\cdots \mathrm{Tr}(A_k)I
  \end{align}
  is positive and multilinear, where $\mathrm{Tr}$ is the canonical trace.
\end{example}
\begin{example}\label{ex2}
Define the mapping $\Phi:\mathcal{M}_p^2\to\mathcal{M}_p$ by $\Phi(T,S)x=(Tx\otimes Sx)x$, where $(u\otimes v)(w)=\langle w,v\rangle u$ for $u,v,w\in {\Bbb C}^p$.  Then $\Phi$ is multilinear. Moreover, if $T$ and $S$ are positive matrices, then
$$\langle \Phi(T,S) x,x\rangle=\langle (Tx\otimes Sx)x,x\rangle=\langle \langle Sx,x\rangle Tx,x\rangle=\langle Sx,x\rangle\langle Tx,x\rangle\geq 0$$
for every $x\in {\Bbb C}^p$. Therefore, $\Phi$ is a positive multilinear mapping.
\end{example}

It would be useful to see that there is a relevant connection between positive linear and multilinear mappings. Corresponding to  every positive multilinear mapping  $\Phi:\mathcal{M}_q^k \to \mathcal{M}_p$, the map  $\Psi:\mathcal{M}_q\to\mathcal{M}_p$ defined by $\Psi(X)=\Phi(X,I,\cdots,I)$ is positive and linear. Conversely, if $\Psi:\mathcal{M}_q\to \mathcal{M}_p$ is positive and linear, then
$$\Phi:\mathcal{M}_q^k\to\mathcal{M}_p,\qquad \Phi(A_1,\cdots, A_k)=\Psi(A_1)\otimes A_2\otimes\cdots\otimes A_k$$
is positive and multilinear.\par

 We state basic properties for positive multilinear mappings:
\begin{lemma}
Every positive multilinear mapping  $\Phi:\mathcal{M}_q^k\to \mathcal{M}_p$ is adjoint-preserving; i.e.,
\[
\Phi(T_1,\ldots,T_k)^* = \Phi(T_1^*,\ldots,T_k^*)
\]
for all $T_i\in \mathcal{M}_q$ $(i=1,\ldots,k)$.
\end{lemma}

\begin{proof}
It suffices to show that $\Phi(A_1,\ldots,A_k)$ is Hermitian for every Hermitian $A_1,\ldots,A_k$. Every Hermitian matrix $A_i$ has a Jordan decomposition
\[
A_i = A_{i+}-A_{i-} \quad \mbox{where $A_{i\pm} \geq 0$}.
\]
Then we have
\begin{align*}
& \Phi(A_1,\ldots,A_k) \\
& = \Phi(A_{1+}-A_{1-},\ldots,A_{k+}-A_{k-}) \\
& = \Phi(A_{1+},\ldots,A_{k+}) - \Phi(A_{1-},A_{2+},\ldots,A_{k+}) + \cdots +(-1)^k \Phi(A_{1-},\ldots,A_{k-})
\end{align*}
and each term is positive and hence $\Phi(A_1,\ldots,A_k)$ is Hermitian.
\end{proof}
The next Lemma is easy to verify.
\begin{lemma}
Every positive multilinear mapping  $\Phi:\mathcal{M}_q^k\to \mathcal{M}_p$ is monotone, i.e.,
\begin{align*}
0\leq A_i\leq & B_i \quad \mbox{for all $i=1,\ldots,k$}\\
& \Longrightarrow \Phi(A_1,A_2,\ldots,A_k) \leq \Phi(B_1,B_2,\ldots,B_k).
\end{align*}
\end{lemma}

The norm of a   multilinear mapping $\Phi:\mathcal{M}_q^k\to \mathcal{M}_p$  is defined by\[
\| \Phi\| = \sup_{\| A_1\|=\cdots=\|A_k\|=1} \| \Phi(A_1, \ldots, A_k)\| = \sup_{\| A_1 \| \leq 1,\cdots,\|A_k\|\leq1} \| \Phi(A_1, \ldots, A_k) \|.
\]
A version of Russu-Dye theorem (see for example \cite[Theorem 2.3.7]{Bh}) has been proved in \cite{Bh1} as follows:
\begin{theorem} \cite{Bh1}\label{thm-RD}
If $\Phi:\mathcal{M}_q\rightarrow\mathcal{M}_p$ is a unital positive multilinear mapping, then $\| \Phi \|=1$.
\end{theorem}

\section{Jensen type inequalities} \label{sec-3}

 If $\Phi$ is a unital positive linear mapping, then the Choi--Davis--Jensen inequality
 \begin{align}\label{CDJ1}
\Phi(f(A))\geq f(\Phi(A))
 \end{align}
  holds for every matrix convex function $f$ on $J$ and  every Hermitian matrix $A$ whose spectrum is contained in $J$. However, this inequality does not hold for a unital positive multilinear mapping in general. For example,  consider the matrix convex function  $f(t)=t^2-t$ and the unital positive multilinear mapping $\Phi(A,B)=A\circ B$. If $A=2I$ and $B=I$, then $2I=f(\Phi(A,B))\nleq\Phi(f(A),f(B))=0$. Another example is the unital positive multilinear mapping $\Psi(A,B)=\frac{1}{p}{\rm Tr}(AB)$ for $A,B\in \mathcal{M}_p$ and $f(\Psi(A,B))\not\leq \Psi(f(A),f(B))$. We consider  an addition hypothesis  on  $f$ under which this inequality holds. For this end, assume that the positive multilinear mapping $\Phi:\mathcal{M}^2_q\rightarrow\mathcal{M}_p$ is defined by $\Phi(A,B)=A\circ B$. If inequality \eqref{CDJ1} is valid for a matrix convex function $f$, then $f(A\circ B)\leq f(A)\circ f(B)$. In particular, $f(xy)\leq f(x)f(y)$ for all $x,y$ in domain of $f$. This turns out to be  a necessary and sufficient condition under which \eqref{CDJ1} holds for every unital positive multilinear mapping.

\begin{definition}
Let $J$ be an interval in ${\Bbb R}$. A real valued continuous function $f$ is super-multiplicative on $J$ if $f(xy)\geq f(x)f(y)$ for all $x,y\in J$, while  $f$ is sub-multiplicative  if $f(xy)\leq f(x)f(y)$ for all $x,y\in J$.
\end{definition}

We present a Choi--Davis--Jensen inequality for positive multilinear mappings.
\begin{theorem} \label{thm-CDJ}
Let $A_i\in\mathcal{M}_q$  $(i=1,\ldots,k)$ be positive matrices and  $\Phi:\mathcal{M}^k_q\rightarrow\mathcal{M}_p$  be a unital positive multilinear mapping. If $f:[0,\infty ) \to\mathbb{R}$ is a sub-multiplicative matrix convex function (resp. a super-multiplicative matrix concave function), then
\begin{align*}
& f(\Phi(A_1,\ldots,A_k)) \leq \Phi(f(A_1),\ldots,f(A_k)) \\
& \qquad \qquad ( resp. \qquad f(\Phi(A_1,\ldots,A_k)) \geq \Phi(f(A_1),\ldots,f(A_k)) ).
\end{align*}
\end{theorem}

\begin{proof}
Assume that $A_i=\sum_{j=1}^{q_i}\lambda_{ij}P_{ij}$ $(i=1,\cdots, k)$ is the spectral decomposition of each $A_i$ for which $\sum_{j=1}^{q_i}P_{ij}=I$. Then
{\small\begin{align}\label{qq3}
\Phi(f(A_1), f(A_2),\cdots, f(A_k))&=\Phi\left(\sum_{j=1}^{q_1}f(\lambda_{1j})P_{1j},\cdots,\sum_{j=1}^{q_k}f(\lambda_{kj})P_{kj}\right)\\
&=\sum_{j_1}^{q_1}\sum_{j_2}^{q_2}\cdots\sum_{j_k}^{q_k}f(\lambda_{1j_1})f(\lambda_{2j_2})\cdots f(\lambda_{kj_k})\Phi(P_{1j_1},\cdots, P_{kj_k}),\nonumber
\end{align}}
by multilinearity of $\Phi$. With  $C(j_1,\cdots, j_k):=\Phi(P_{1j_1},\cdots, P_{kj_k})^{\frac{1}{2}}$  we get
{\small\begin{align*}
f&\left(\Phi(A_1,\cdots, A_k)\right)\\
&=f\left(\sum_{j_1}^{q_1}\sum_{j_2}^{q_2}\cdots\sum_{j_k}^{q_k}\lambda_{1j_1}\lambda_{2j_2}\cdots \lambda_{kj_k}\Phi(P_{1j_1},\cdots, P_{kj_k})\right)\quad (\mbox{by multilinearity of $\Phi$})\\
&=f\left(\sum_{j_1}^{q_1}\sum_{j_2}^{q_2}\cdots\sum_{j_k}^{q_k}C(j_1,\cdots, j_k)\lambda_{1j_1}\lambda_{2j_2}\cdots \lambda_{kj_k}C(j_1,\cdots, j_k)\right)\\
&\leq \sum_{j_1}^{q_1}\sum_{j_2}^{q_2}\cdots\sum_{j_k}^{q_k}C(j_1,\cdots, j_k)f(\lambda_{1j_1}\lambda_{2j_2}\cdots \lambda_{kj_k})C(j_1,\cdots, j_k)\quad (\mbox{by matrix convexity of $f$})\\
&\leq \sum_{j_1}^{q_1}\sum_{j_2}^{q_2}\cdots\sum_{j_k}^{q_k} f(\lambda_{1j_1})f(\lambda_{2j_2})\cdots f(\lambda_{kj_k})\Phi(P_{1j_1},\cdots, P_{kj_k}) \ (\mbox{by  $f(xy)\leq f(x)f(y)$ on $[0,\infty)$})\\
&=\Phi(f(A_1),\cdots, f(A_k)).
\end{align*}}
\end{proof}

As an example,   the power functions $f(t)=t^r$  are   sub-multiplicative and so  we have the following
 extension for  Kadison and  Choi's   inequalities.

\begin{corollary}\label{co5}
Suppose that $\Phi:\mathcal{M}_q^k\rightarrow\mathcal{M}_p$ is a unital positive multilinear mapping.
\begin{itemize}
\item[(1)] If $0\leq r\leq 1$, then
\begin{eqnarray*}
\Phi(A^r_1,\cdots, A^r_k)\leq\Phi(A_1,\cdots, A_k)^r\qquad (A_i\geq0),
\end{eqnarray*}
\item[(2)] If $-1\leq r\leq0$ and $1\leq r\leq 2$, then
\begin{eqnarray*}
\Phi(A_1,\cdots, A_k)^r\leq\Phi(A^r_1,\cdots, A^r_k) \qquad (A_i>0).
\end{eqnarray*}
\end{itemize}
\end{corollary}

\begin{corollary}
Let $\Phi:\mathcal{M}_q^k\rightarrow\mathcal{M}_p$ be a uninal positive multilinear map. If either $s\leq t$, $s\not\in (-1,1), t\not\in (-1,1)$ or $\frac{1}{2}\leq s\leq 1\leq t$ or $s\leq -1\leq t\leq -\frac{1}{2}$, then
\begin{eqnarray*}
\Phi(A^s_1,\cdots, A^s_k)^{\frac{1}{s}}\leq\Phi(A^t_1,\cdots, A^t_k)^{\frac{1}{t}}
\end{eqnarray*}
for all positive matrices $A_1,\cdots, A_k$.
\end{corollary}
\begin{proof}
If $0\leq s\leq t$, then  $r=\frac{s}{t}\leq 1$.
 Applying Corollary \ref{co5} with $A_i^t$ instead of $A_i$  now implies that \begin{eqnarray*}
\Phi(A^s_1,\cdots, A^s_k)\leq\Phi(A^t_1,\cdots, A^t_k)^{\frac{s}{t}}.
\end{eqnarray*}
Applying the L$\ddot{o}$wner-Heinz inequality with the power $\frac{1}{s}$ yields that
\begin{eqnarray*}
\Phi(A^s_1,\cdots, A^s_k)^{\frac{1}{s}}\leq\Phi(A^t_1,\cdots, A^t_k)^{\frac{1}{t}}.
\end{eqnarray*}
\end{proof}

\begin{example}
The matrix concave function $f(t)=\log t$ defined on $(0,\infty)$ is super-multiplicative on $[1,e^2]$. For if $x,y\in[1,e^2]$, then $\log x,\log y\in [0,2]$ and so $(\log x)(\log y)\leq 4$. Hence
$(\log x)^2(\log y)^2\leq 4(\log x)(\log y)$
and so
$$f(x)f(y)=\log x \log y\leq 2 \sqrt{(\log x)(\log y)}\leq \log x+\log y=\log (xy)=f(xy),$$
where we used the arithmetic-geometric mean inequality. Now applying Theorem \ref{thm-CDJ} we get
$$\log (A_1\otimes A_2)\geq \log (A_1)\otimes \log (A_2)$$
for all positive matrices $A_i$ with eigenvalues in $[1,e^2]$. However the inequality $\log(xyz)\leq \log x \log y \log z$ does not hold for $x,y,x\in [1,e^2]$ in general.
\end{example}

We recall the theory of matrix means by Kubo--Ando\cite{KA1}. A key for the theory is that there is a one-to-one correspondence between the matrix mean $\sigma$ and the nonnegative matrix monotone function $f(t)$ on $(0,\infty )$ with $f(1)=1$ by the formula
\[
f(t)=1 \ \sigma \ t \qquad \mbox{for all $t>0$},
\]
or
\[
A\ \sigma \ B = A^{\frac{1}{2}}f(A^{-\frac{1}{2}}BA^{-\frac{1}{2}})A^{\frac{1}{2}} \quad \mbox{for all $A,B>0$.}
\]
We say that $f$ is the representing function for $\sigma$. In this case, notice that $f(t)$ is matrix monotone if and only if it is matrix concave \cite{FMPS}. Moreover, every matrix mean $\sigma$ is subadditive: $A\ \sigma\ C+B\ \sigma \ D \leq (A+B)\ \sigma \ (C+D)$.
By a theorem of Ando \cite{An1}, if  $A$ and $B$ are positive matrices and $\Phi$ is a strictly positive linear mapping, then for each $\a \in [0,1]$
\begin{equation} \label{eq:Ando}
\Phi(A\s_{\a} B) \leq \Phi(A)\s_{\a} \Phi(B),
\end{equation}
where the $\a$-geometric matrix mean is defined by
\[
A\s_{\a} B = A^{\frac{1}{2}}\left( A^{-\frac{1}{2}}BA^{-\frac{1}{2}} \right)^{\a}A^{\frac{1}{2}}
\]
and the geometric matrix mean $\s$ is defined by $\s = \s_{1/2}$, namely
\[
A\s B = A^{\frac{1}{2}}\left( A^{-\frac{1}{2}}BA^{-\frac{1}{2}} \right)^{\frac{1}{2}}A^{\frac{1}{2}}.
\]
By  aid  of  Theorem \ref{thm-CDJ}, we show the positive multilinear mapping version of Ando's inequality \eqref{eq:Ando}.

\begin{proposition} \label{thm-s}
Let $\sigma$ be a matrix mean with a super-multiplicative representing function $f:[0,\infty)\rightarrow[0,\infty)$, and $\Phi:\mathcal{M}_q^k\rightarrow\mathcal{M}_p$ be a strictly positive multilinear mapping. Then
\begin{align*}\label{Eq}
\Phi(A_1\sigma B_1,\cdots , A_k\sigma B_k)\leq\Phi(A_1,\cdots , A_k)\sigma\Phi(B_1,\cdots , B_k)
\end{align*}
for all $A_1,\cdots,A_k>0$ and $B_1,\cdots, B_k\geq0$.
\end{proposition}
\begin{proof}
For  $A_1,\cdots, A_k>0$ assume that a map $\Psi:\mathcal{M}_q^k\rightarrow\mathcal{M}_p$ is defined by
\begin{eqnarray*}
\Psi(X_1,\cdots, X_k):=\Phi(A_1,\cdots, A_k)^{-\frac{1}{2}}\Phi\left(A^{\frac{1}{2}}_1X_1A^{\frac{1}{2}}_1,\cdots , A^{\frac{1}{2}}_kX_kA^{\frac{1}{2}}_k\right)\Phi(A_1,\cdots, A_K)^{-\frac{1}{2}}.
\end{eqnarray*}
Clearly $\Psi$ is  unital strictly positive and  multilinear. For every $i=1,\cdots,k$, put $Y_i:=A^{-\frac{1}{2}}_iB_iA_i^{-\frac{1}{2}}$\ $(i=1,\cdots k)$.
 Then it follows from matrix concavity of $f$ and Theorem~\ref{thm-CDJ} that
\begin{align*}
\Phi(A_1,\cdots, A_k)^{-\frac{1}{2}}&\Phi (A_1\sigma B_1,\cdots , A_k\sigma B_k)\Phi(A_1,\cdots, A_k)^{-\frac{1}{2}}\\
&=\Psi\left(f(Y_1),\cdots, f(Y_k)\right)\\
&\leq f\big(\Psi(Y_1,\cdots, Y_k)\big)\\
&=f\left(\Phi(A_1,\cdots, A_k)^{-\frac{1}{2}}\Phi(B_1,\cdots, B_k)\Phi(A_1,\cdots, A_k)^{-\frac{1}{2}}\right),
\end{align*}
which completes the proof.
\end{proof}

For a matrix mean $\sigma$ with the representing function $f$, assume that $\sigma^{o}$, the transpose of $\sigma$  is defined by $A\ \sigma^{o} \ B=B \ \sigma \ A$. In this case, the representing function of $\sigma^{o}$ is $f^{o}$ defined by $f^{o}(t)=tf(1/ t)$. For $\alpha\in[0,1]$, the transpose of geometric mean $\sharp_\alpha$ turns out to be $\sharp_{1-\alpha}$.  Ando \cite{An1} showed that if $A,B>0$, then
\[
A \ \circ \ B \geq (A\ \s_{\a} \ B)\ \circ \ ( A\ \s_{1-\a}\ B)
\]
for every $\alpha\in[0,1]$.   We state a generalization of this result using positive multilinear mappings.
\begin{corollary} \label{co-ML-OP}
Let $\sigma$ be a matrix mean with the representing function $f$ which is super-multiplicative and not affine. Let $A,B$ be positive matrices. If $\Phi:\mathcal{M}_q^2\rightarrow\mathcal{M}_p$ is a unital positive multilinear mapping, then
\[
\Phi(A,B) + \Phi(B,A) \geq \Phi(A \ \sigma \ B, A \ \sigma^{o} \ B) + \Phi(A \ \sigma^{o}  \ B, A \ \sigma \ B).
\]
In particular, for each $\a \in [0,1]$
\[
\Phi(A,B)+\Phi(B,A) \geq \Phi(A \ \s_{\a} \ B, A \ \s_{1-\a} \ B) + \Phi(A \ \s_{1-\a} \ B, A \ \s_{\a} \ B).
\]
\end{corollary}

\begin{proof}
It follows from subadditivity of $\sigma$ and Theorem~\ref{thm-s} that
\begin{align*}
\Phi(A,B) + \Phi(B,A) & = \left[ \Phi(A,B)+\Phi(B,A) \right] \sigma \left[ \Phi(B,A)+\Phi(A,B) \right] \\
& \geq \Phi(A,B) \sigma \Phi(B,A) + \Phi(B,A) \sigma \Phi(A,B) \\
& \geq \Phi(A\sigma B, B\sigma A) + \Phi(B\sigma A, A\sigma B) \\
& =  \Phi(A \sigma B, A \sigma^{o} B) + \Phi(A \sigma^{o} B, A \sigma B).
\end{align*}
\end{proof}


If $A$ is positive, then $A\circ A^{-1}\geq I$. This inequality is known as  Fiedler's inequality, see  \cite{Bh,Fiedler}. As an application of  Theorem \ref{thm-s} , we have the following extension of  Fiedler's inequality by Aujla-Vasudeva \cite{AV1}.
\begin{proposition}
Let $\Phi:\mathcal{M}_q^2\rightarrow\mathcal{M}_p$ be a positive multilinear map. If $\alpha,\beta\in\mathbb{R}$ and $\lambda\in[0,1]$, then
\begin{eqnarray*}
\Phi(A^{\alpha},A^{\beta})+\Phi(A^{\beta}, A^{\alpha})\geq \Phi(A^{(1-\lambda)\alpha+\lambda\beta},A^{(1-\lambda)\beta+\lambda\alpha}) + \Phi(A^{(1-\lambda)\beta+\lambda\alpha},A^{(1-\lambda)\alpha+\lambda\beta})
\end{eqnarray*}
for all matrices $A>0$.
\end{proposition}
\begin{proof}
We have
\begin{align*}
& \Phi(A^{\alpha},A^{\beta})+\Phi(A^{\beta}, A^{\alpha}) \\
&=\left(\Phi(A^{\alpha}, A^{\beta})+\Phi(A^{\beta}, A^{\alpha})\right)\sharp_{\l}\left(\Phi(A^{\beta}, A^{\alpha})+\Phi(A^{\alpha}, A^{\beta})\right)\\
&\geq \Phi(A^{\alpha}, A^{\beta})\sharp_{\l}\Phi(A^{\beta}, A^{\alpha})+\Phi(A^{\beta}, A^{\alpha})\sharp_{\l}\Phi(A^{\alpha}, A^{\beta})\\
&\geq \Phi(A^{\alpha}\sharp_{\l} A^{\beta}, A^{\beta}\sharp_{\l} A^{\alpha})+\Phi(A^{\beta}\sharp_{\l} A^{\alpha}, A^{\alpha}\sharp_{\l} A^{\beta})\\
&= \Phi(A^{(1-\lambda)\alpha+\lambda\beta},A^{(1-\lambda)\beta+\lambda\alpha}) + \Phi(A^{(1-\lambda)\beta+\lambda\alpha},A^{(1-\lambda)\alpha+\lambda\beta}).
\end{align*}
\end{proof}
With $\a=1$ and $\b=-1$ and $\lambda=\frac{1}{2}$ we get the following result.
\begin{corollary}\label{lem4}
Let $\Phi:\mathcal{M}_q^2\rightarrow\mathcal{M}_p$ be a strictly positive multilinear map. If $A>0$, then
\begin{eqnarray*}
\Phi(A, A^{-1})+\Phi(A^{-1}, A)\geq 2I_p.
\end{eqnarray*}
\end{corollary}
The next lemma will be used in the sequel.
\begin{lemma}\cite[Theorem 1.3.3]{Bh}\label{lem1}
 Let A,B be strictly positive matrices. Then the block matrix
$\left[
\begin{array}{cc}
A & X \\
X^* & B
\end{array}\right]$
is positive if and only if $A\geq XB^{-1}X^*$.
\end{lemma}
In \cite{Choi}, Choi generalized Kadison's inequality to normal matrices by showing that if $\Phi$ is a  unital positive linear mapping, then
\[
\Phi(A)\Phi(A^*)\leq \Phi(A^*A) \quad \mbox{and} \quad \Phi(A^*)\Phi(A) \leq \Phi(A^*A).
\]
for every normal matrix $A$. A similar result holds true for positive multilinear mappings.

\begin{proposition}\label{CML}
Let $\Phi:\mathcal{M}_q^k\rightarrow\mathcal{M}_p$ be a positive multilinear map. Then
\begin{align*}
& \Phi(A^*_1A_1, A_2^*A_2,\cdots , A^*_kA_k)\geq\Phi(A_1,\cdots, A_k)\Phi(A_1^*,\cdots , A_k^*) \\
& \Phi(A^*_1A_1, A_2^*A_2,\cdots , A^*_kA_k)\geq\Phi(A_1^*,\cdots, A_k^*)\Phi(A_1,\cdots , A_k)
\end{align*}
for all normal matrices $A_1, A_2,\cdots, A_k$.
\end{proposition}
\begin{proof}
Since the matrix $Z_{\mu}=\left(\begin{array}{cc}|\mu|^2&\mu\\\bar{\mu}&1\end{array}\right)$ is positive for all $\mu\in\mathbb{C}$, it follows that
\begin{eqnarray*}
\left(\begin{array}{cc}|\lambda_{1j_1}\cdots\lambda_{kj_k}|^2&\lambda_{1j_1}\cdots\lambda_{kj_k}\\\overline{\lambda_{1j_1}\cdots\lambda_{kj_k}}&1\end{array}\right)=Z_{\lambda_{1j_1}}\circ\cdots\circ Z_{\lambda_{kj_k}}
\end{eqnarray*}
is positive. If $A_i=\sum_{j=1}^{q}\lambda_{ij}P_{ij}$  is the spectral decomposition of each $A_i$, then
\begin{eqnarray*}
\Phi(A^*_1A_1, A^*_2A_2,\cdots , A^*_k A_k)&=&\Phi\left(\sum_{j_1=1}^{q}|\lambda_{1j_1}|^2P_{1j_1},\cdots,\sum_{j_k=1}^{q}|\lambda_{kj_k}|^2P_{kj_k}\right)\\
&=&\sum_{j_1=1}^{q}\sum_{j_2=1}^{q}\cdots\sum_{j_k=1}^{q}|\lambda_{1j_1}\lambda_{2j_2}\cdots\lambda_{kj_k}|^2\Phi(P_{1j_1},\cdots , P_{kj_k})
\end{eqnarray*}
and
\begin{eqnarray*}
\Phi(A_1,\cdots , A_k)=\sum_{j_1=1}^{q}\sum_{j_2=2}^{q}\cdots\sum_{j_k=1}^{q}\lambda_{1j_1}\cdots\lambda_{kj_k}\Phi(P_{1j_1},\cdots , P_{kj_k})\\
\Phi(A^*_1,\cdots , A^*_k)=\sum_{j_1=1}^{q}\sum_{j_2=2}^{q}\cdots\sum_{j_k=1}^{q}\overline{\lambda_{1j_1}\cdots\lambda_{kj_k}}\Phi(P_{1j_1},\cdots , P_{kj_k}).
\end{eqnarray*}
It follows that
{\small\begin{align*}
&\left(\begin{array}{cc}\Phi(A^*_1A_1,\cdots , A^*_kA_k)
&\Phi(A_1,\cdots, A_k)\\ \Phi(A^*_1,\cdots , A^*_k)&1\end{array}\right)\\
&\qquad\qquad\hspace{3cm} =\sum_{j_1=1}^{q}\cdots\sum_{j_k=1}^{q}
\left(Z_{\lambda_{1j_1}}\circ\cdots \circ Z_{\lambda_{kj_k}}\right)\otimes\Phi(P_{1j_1}\cdots , P_{kj_k})\geq 0.
\end{align*}}
 Lemma~\ref{lem1} then implies that
\begin{eqnarray*}
\Phi(A^*_1A_1, A^*A_2,\cdots , A^*_kA_k)&\geq&\Phi(A_1, A_2,\cdots, A_k)\Phi(A^*_1, A^*_2,\cdots, A^*_k).
\end{eqnarray*}
\end{proof}

Ando \cite{An1} unified Kadison and Choi inequalities into a single form as follows:
Let $\Phi$ be a strictly positive linear mapping on $\mathcal{M}_p$. Then
\begin{equation} \label{eq:A1}
\Phi(HA^{-1}H) \geq \Phi(H)\Phi(A)^{-1}\Phi(H)
\end{equation}
whenever $H$ is Hermitian and $A>0$. Moreover, it is known a more general version of \eqref{eq:A1} in \cite[Proposition 2.7.5]{Bh}:
\begin{align}\label{ee}
A\geq X^*A^{-1}X \Longrightarrow \Phi(A) \geq \Phi(X)^*\Phi(A)^{-1}\Phi(X)
\end{align}
for $X$ is arbitrary and $A>0$.\par
As another application of Theorem \ref{thm-CDJ}  we present a  multilinear map version of \eqref{eq:A1} and \eqref{ee}.
\begin{proposition}
Let $\Phi$ be a strictly positive multilinear mapping. Then
\[
\Phi(H_1,\ldots,H_k)\Phi(A_1,\ldots,A_k)^{-1}\Phi(H_1,\ldots,H_k) \leq \Phi(H_1A_1^{-1}H_1,\ldots,H_kA_k^{-1}H_k)
\]
whenever $H_i$ is Hermitian and $A_i>0$ \ $(i=1,\ldots,k)$.
\end{proposition}

\begin{proof}
Assume that the positive multilinear mapping $\Psi$  is defined by
\begin{equation} \label{eq:psi}
\Psi(Y_1,\ldots,Y_k) = \Phi(A_1,\ldots,A_k)^{-\frac{1}{2}}\Phi(A_1^{\frac{1}{2}}Y_1A_1^{\frac{1}{2}},\ldots,A_k^{\frac{1}{2}}Y_kA_k^{\frac{1}{2}})\Phi(A_1,\ldots,A_k)^{-\frac{1}{2}}.
\end{equation}
 By Corollary~\ref{co5} we have
\[
\Psi(Y_1^2,\ldots,Y_k^2)\geq \Psi(Y_1,\ldots,Y_k)^2
\]
for every Hermitian $Y_i$. Choose $Y_i=A_i^{-\frac{1}{2}}H_iA_i^{-\frac{1}{2}}$ to get
\[
\Psi(A_1^{-\frac{1}{2}}H_1A_1^{-1}H_1A_1^{-\frac{1}{2}},\ldots, A_k^{-\frac{1}{2}}H_kA_k^{-1}H_kA_k^{-\frac{1}{2}}) \geq \Psi(A_1^{-\frac{1}{2}}H_1A_1^{-\frac{1}{2}},\ldots,A_k^{-\frac{1}{2}}H_kA_k^{-\frac{1}{2}})^2.
\]
Multiplying both sides by $\Phi(A_1,\ldots,A_k)^{\frac{1}{2}}$ we conclude the desired inequality.
\end{proof}

\begin{proposition}
Let $\Phi$ be a strictly positive multilinear mapping. If $A_i>0$ and $X_i\in \mathcal{M}_q$, then
\begin{align*}
A_i\geq & X_i^*A_i^{-1}X_i \    (i=1,\ldots,k) \Longrightarrow \\
& \Phi(A_1.\ldots,A_k) \geq \Phi(X_1,\ldots,X_k)^* \Phi(A_1,\ldots,A_k)^{-1} \Phi(X_1,\ldots,X_k).
\end{align*}
\end{proposition}

\begin{proof}
Let $\Psi$ be the positive multilinear mapping defined by \eqref{eq:psi}. By Theorem~\ref{CML}
\[
Y_i^*Y_i \leq I \   (i=1,\ldots,k)\  \Longrightarrow \  \Psi(Y_1,\ldots,Y_k)^*\Psi(Y_1,\ldots,Y_k) \leq I.
\]
Let $A_i\geq X_i^*A_i^{-1}X_i$ and put $Y_i=A_i^{-\frac{1}{2}}X_iA_i^{-\frac{1}{2}}$ \ $(i=1,\ldots,k)$. Then $Y_i^*Y_i = A_i^{-\frac{1}{2}}X_i^*A_i^{-1}X_iA_i^{-\frac{1}{2}} \leq I$. Hence
\[
\Psi(A_1^{-\frac{1}{2}}X_1^*A_1^{-\frac{1}{2}},\ldots,A_k^{-\frac{1}{2}}X_k^*A_k^{-\frac{1}{2}})\Psi(A_1^{-\frac{1}{2}}X_1A_1^{-\frac{1}{2}},\ldots,A_k^{-\frac{1}{2}}X_kA_k^{-\frac{1}{2}})\leq I
\]
Multiply both sides by $\Phi(A_1,\ldots,A_k)^{\frac{1}{2}}$ to obtain the result.
\end{proof}

\par
\medskip

\section{Kantorovich inequality and convex theorems}

Let $w_i$ be positive scalars. If $m\leq a_i\leq M$ for some positive scalars $m$ and $M$, then the Kantorovich inequality
 \begin{align}\label{kant}
 \left(\sum_{i=1}^{n}w_ia_i\right)\left(\sum_{i=1}^{n}w_i a_i^{-1}\right)\leq \frac{(m+M)^2}{4mM}\left(\sum_{i=1}^{n}w_i\right)^2
\end{align}
holds. There are several operator version of this inequality, see e.g. \cite{MA}.  For example if $A$ is strictly positive matrix with $0<m\leq A\leq M$ for some  positive scalars $m<M$, then
\begin{align}\label{ka1}
 \langle A^{-1}x,x\rangle \leq \frac{(m+M)^2}{4mM}\langle Ax,x\rangle^{-1}\qquad (x\in {\Bbb C}^p,\ \|x\|=1).
\end{align}
We would like to refer the reader to \cite{M} to find a recent  survey concerning   operator Kantorovich inequality.

 We show a matrix version of the Kantorovich inequality \eqref{kant} including positive multilinear mappings: If $t_1,\cdots,t_n$ are positive scalars and  $0<m\leq A_i\leq M$ \ $(i=1,\ldots,k)$,  then
\begin{align}\label{km}
  \Phi\left(\sum_{i=1}^{n}t_iA_i,\sum_{i=1}^{n}t_iA_i^{-1}
\right)\leq\frac{m^2+M^2}{2mM}\Phi\left(\sum_{i=1}^{n}t_i,\sum_{i=1}^{n}t_i \right).
\end{align}

Another  particular case is $\Phi(A,B)=A\circ B$ and $n=1$ which gives $A\circ A^{-1}\leq\frac{m^2+M^2}{2mM}$, see \cite{MPS}. First we give a more general form of \eqref{km}.

\begin{proposition}\label{pr2}
Let $\Phi:\mathcal{M}_q^2\to \mathcal{M}_p$  be a unital positive multilinear map.
If $A_i,B_i\in\mathcal{M}_q$  are positive  matrices with $m\leq A_i,B_i\leq M$ \ $(i=1,\cdots,n)$, then
{\small\begin{align}\label{qq2}
\Phi\left(\sum_{i=1}^{n}X_i^{\frac{1}{2}}A_iX_i^{\frac{1}{2}},\sum_{i=1}^{n}Y_i^{\frac{1}{2}}B_i^{-1}
Y_i^{\frac{1}{2}}\right)
&+\Phi\left(\sum_{i=1}^{n}X_i^{\frac{1}{2}}A_i^{-1}X_i^{\frac{1}{2}},\sum_{i=1}^{n}Y_i^{\frac{1}{2}}B_i
Y_i^{\frac{1}{2}}\right)\nonumber\\
&\qquad
\leq\frac{m^2+M^2}{mM}\ \Phi\left(\sum_{i=1}^{n}X_i,\sum_{i=1}^{n}Y_i\right)
\end{align}}
for all positive  matrices $X_i,Y_i$.
\end{proposition}

\begin{proof}
  Let $A_i$ and $B_j$ are positive matrices with eigenvalues $\lambda_{i1},\cdots,\lambda_{iq}$ and $\mu_{j1},\cdots,\mu_{jq}$, respectively, so that $m\leq\lambda_{ik}\leq M$ and $m\leq\mu_{jk}\leq M$ for every $k=1,\cdots, q$. Hence
  \begin{align}\label{qq1}
\lambda_{ik}\mu_{jk}^{-1}+\lambda_{ik}^{-1}\mu_{jk}\leq \frac{m^2+M^2}{mM}.
  \end{align}
If $A_i=\sum_{k=1}^{q}\lambda_{ik}P_{ik}$ and $B_j=\sum_{\ell=1}^{q}\mu_{j\ell}P_{j\ell}$, then
   \begin{align*}
    &\Phi\left(X_i^{\frac{1}{2}}A_iX_i^{\frac{1}{2}},Y_j^{\frac{1}{2}}B_j^{-1}Y_j^{\frac{1}{2}}\right)+
    \Phi\left(X_i^{\frac{1}{2}}A_i^{-1}X_i^{\frac{1}{2}},Y_j^{\frac{1}{2}}B_jY_j^{\frac{1}{2}}\right)\\
     &=
     \Phi\left(\sum_{k=1}^{q}\lambda_{ik}X_i^{\frac{1}{2}}P_{ik}X_i^{\frac{1}{2}},
     \sum_{\ell=1}^{q}\mu_{j\ell}^{-1}Y_j^{\frac{1}{2}}Q_{j\ell}Y_j^{\frac{1}{2}} \right)
     + \Phi\left(\sum_{k=1}^{q}\lambda_{ik}^{-1}X_i^{\frac{1}{2}}P_{ik}X_i^{\frac{1}{2}},
     \sum_{\ell=1}^{q}\mu_{j\ell}Y_j^{\frac{1}{2}}Q_{j\ell}Y_j^{\frac{1}{2}} \right)
          \\
     &=\sum_{k=1}^{q}\sum_{\ell=1}^{q}(\lambda_{ik}\mu_{j\ell}^{-1}+\lambda_{ik}^{-1}\mu_{j\ell})
     \Phi\left(X_i^{\frac{1}{2}}P_{ik}X_i^{\frac{1}{2}},Y_j^{\frac{1}{2}}Q_{j\ell}Y_j^{\frac{1}{2}}\right)\\
     &\leq \frac{m^2+M^2}{mM}\sum_{k=1}^{q}\sum_{\ell=1}^{q}
     \Phi\left(X_i^{\frac{1}{2}}P_{ik}X_i^{\frac{1}{2}},Y_j^{\frac{1}{2}}Q_{j\ell}Y_j^{\frac{1}{2}}\right)\\
     &=\frac{m^2+M^2}{mM}\Phi(X_i,Y_j),
        \end{align*}
        where the inequality follows from \eqref{qq1}. The multilinearity of $\Phi$ then can be applied to  show \eqref{qq2}.
\end{proof}
With the assumption as in Proposition \ref{pr2}, if we put $B=A$ and $Y=X$, then we get
\begin{align*}
\Phi\left(\sum_{i=1}^{n}X_i^{\frac{1}{2}}A_iX_i^{\frac{1}{2}},\sum_{i=1}^{n}X_i^{\frac{1}{2}}A_i^{-1}
X_i^{\frac{1}{2}}\right)\leq\frac{m^2+M^2}{2mM}\ \Phi\left(\sum_{i=1}^{n}X_i,\sum_{i=1}^{n}X_i\right).
\end{align*}
If $t_1,\cdots,t_n$ are positive scalars, then with $X_i=t_i$ we get \eqref{km}.

Let $\Phi:\mathcal{M}_q^2\to \mathcal{M}_p$   be defined as Example \ref{ex2}. Assume that $n=1$ and $X=Y=I$. If $0<m\leq A,B\leq M$, then as another  consequence of Proposition \ref{pr2} we get another operator Kantorovich inequality as follows:
$$\langle x,B^{-1}x\rangle\langle Ax,x\rangle+\langle x,Bx\rangle\langle A^{-1}x,x\rangle\leq\frac{m^2+M^2}{mM}.$$

Next we prove a convex theorem involving the positive multilinear mappings.

\begin{theorem}\label{th1}
If $\Phi:\mathcal{M}_q^2\to \mathcal{M}_p$ is a positive multilinear mapping, then the function
$$f(t)=\Phi\left(X^*A^{1+t}X,Y^*B^{1-t}Y\right)
+\Phi\left(X^*A^{1-t}X,Y^*B^{1+t}Y\right)$$
is convex on ${\Bbb R}$ and attains its minimum at $t = 0$
for all strictly positive  matrices $A,B\in\mathcal{M}_q$ and all $X,Y \in \mathcal{M}_q$.
\end{theorem}

\begin{proof}
 Assume that $A,B$ are positive  matrices in $\mathcal{M}_q$ with eigenvalues $\lambda_1,\cdots,\lambda_q$ and $\mu_1,\cdots,\mu_q$, respectively. Since $g(s)=x^s+x^{-s}$ ($x>0$) is convex on ${\Bbb R}$, for every $i,j=1,\cdots, q$ and $x=\lambda_i\mu_j^{-1}$ we have
  \begin{equation} \label{q3}
2(\l_i^s\mu_j^{-s}+\l_i^{-s}\mu_j^s ) \leq \l_i^{s+t}\mu_j^{-(s+t)}+\l_i^{-(s+t)}\mu_j^{s+t}+\l_i^{s-t}\mu_j^{-(s-t)}+\l_i^{-(s-t)}\mu_j^{s-t}.
\end{equation}
    Multiply both sides of \eqref{q3} by $\lambda_i\mu_j$ to obtain
  \begin{align} \label{q4}
  2(\l_i^{1+s}\mu_j^{1-s}+\l_i^{1-s}\mu_j^{1+s} ) & \leq \l_i^{1+s+t}\mu_j^{1-(s+t)}+\l_i^{1-(s+t)}\mu_j^{1+s+t} \notag \\
&\ \ \  +\l_i^{1+s-t}\mu_j^{1-(s-t)}+\l_i^{1-(s-t)}\mu_j^{1+s-t}.
\end{align}
Now if $A=\sum_{i=1}^{q}\lambda_iP_i$ and $B=\sum_{j=1}^{q}\mu_jQ_j$ is the spectral decomposition of $A,B$, respectively, then
    {\small\begin{align*}
    2f(s)&=2\left[ \Phi\left(X^*A^{1+s}X,Y^*B^{1-s}Y\right)+
    \Phi\left(X^*A^{1-s}X,Y^*B^{1+s}Y\right) \right]\\
    &=\sum_{i=1}^{q}\sum_{j=1}^{q} 2\left(\lambda_i^{s+1}\mu_j^{-s+1}+\lambda_i^{-s+1}\mu_j^{s+1}\right)
    \Phi\left(X^*P_iX,Y^*Q_jY\right) \\
    &\qquad \qquad\qquad\hspace{6cm}\quad (\mbox{by multilinearity of $\Phi$})\\
    &\leq \sum_{i=1}^{q}\sum_{j=1}^{q} \left(\l_i^{1+s+t}\mu_j^{1-(s+t)}+\l_i^{1-(s+t)}\mu_j^{1+s+t} +\l_i^{1+s-t}\mu_j^{1-(s-t)}+\l_i^{1-(s-t)}\mu_j^{1+s-t} \right) \\
& \quad \quad
    \Phi\left(X^*P_iX,Y^*Q_jY\right) \\
    & \qquad \qquad\qquad\hspace{5cm}\quad(\mbox{by \eqref{q4} and the positivity of $\Phi$})\\
    &=\Phi\left(X^*A^{1+s+t}X,Y^*B^{1-(s+t)}Y\right)    +\Phi\left(X^*A^{1-(s+t)}X,Y^*B^{1+(s+t)}Y\right) + \\
& \qquad \qquad \Phi\left(X^*A^{1+s-t}X,Y^*B^{1-(s-t)}Y\right)
    +\Phi\left(X^*A^{1-(s-t)}X,Y^*B^{1+(s-t)}Y\right) \\
& = f(s+t)+f(s-t).
    \end{align*}}
    Therefore $f$ is convex on ${\Bbb R}$. Since  $f(-t)=f(t)$,  this together with the convexity of $f$ implies that $f$ attains its minimum at $t=0$.
\end{proof}


Two special cases of the last theorem reads as follows.
\begin{corollary}\label{co}
Let $\Phi:\mathcal{M}_q^2\to \mathcal{M}_p$ be a    positive multilinear map. For all $A,B\geq 0$, the function
$$f(t)=\Phi\left(A^{1+t},B^{1-t}\right)+\Phi\left(A^{1-t},B^{1+t}\right)$$
is decreasing on the interval $[-1, 0]$, increasing on the interval $[0, 1]$ and attains its minimum at $t = 0$.
\end{corollary}
\begin{proof}
  Follows from Theorem \ref{th1} with $X=Y=I$.
\end{proof}
Note that applying Corollary \ref{co} with $\Phi(A,B)=A\otimes B$  implies \cite[Theorem 3.2]{MMA}.
\begin{corollary}\label{co2}
 Let $A,B\geq0$.  If $\Phi:\mathcal{M}_q^2\to \mathcal{M}_p$ is a    positive multilinear mapping, then the function
  $$f(t)=\Phi\left(A^{t},B^{1-t}\right)+\Phi\left(A^{1-t},B^{t}\right)$$
  is decreasing on $[0,\frac{1}{2}]$, increasing on $[\frac{1}{2},1]$ and attains its minimum at $t=\frac{1}{2}$.
\end{corollary}
\begin{proof}
  Replace $A,B$ in Theorem \ref{th1} by $A^\frac{1}{2},B^\frac{1}{2}$ and $\frac{1+t}{2}$ by $t$.
\end{proof}
As an example, the mapping $(X,Y)\to \mathrm{Tr}(XY)$ is positive and multilinear. Hence, the function
$$f(t)=\mathrm{Tr}\left(A^{t}B^{1-t}\right)+\mathrm{Tr}\left(A^{1-t}B^{t}\right)$$
is decreasing on $[0,\frac{1}{2}]$, increasing on $[\frac{1}{2},1]$ and has a minimum at $t=\frac{1}{2}$. In other words
$$2\mathrm{Tr}\left(A^\frac{1}{2}B^\frac{1}{2}\right)\leq \mathrm{Tr}\left(A^{t}B^{1-t}\right)+\mathrm{Tr}\left(A^{1-t}B^{t}\right)\leq\mathrm{Tr}(A)+\mathrm{Tr}(B)$$
for all $t\in[0,1]$ and all positive matrices $A$ and $B$.


\section{Reverse of the Choi--Davis--jensen inequality}

In this section, we consider reverse type inequalities of the Choi--Davis--Jensen inequality for positive multilinear mappings. The following lemma is well-known.
\begin{lemma}\cite{MPS} \label{lem-RCDJI}
Let $A_i$ be strictly positive matrices with $m\leq A_i\leq M$  for some scalars $0<m<M$ and let $C_i$ \ $(i=1,\ldots,k)$ be matrices with $\sum_{i=1}^k C_i^*C_i=I$.   If a real valued continuous function $f$ on $[m,M]$ is convex (resp. concave) and $f(t)>0$ on $[m,M]$, then
\[
\sum_{i=1}^k C_i^*f(A_i)C_i \leq \a(m,M,f) f\left( \sum_{i=1}^k C_i^*A_iC_i \right)
\]
\[
\left( {\rm resp.} \quad \b(m,M,f) f\left( \sum_{i=1}^k C_i^*A_iC_i \right) \leq \sum_{i=1}^k C_i^*f(A_i)C_i, \right)
\]
where
{\small\begin{equation} \label{eq:af}
\a(m,M,f) = \max \left\{ \frac{1}{f(t)} \left( \frac{f(M)-f(m)}{M-m}(t-m)+f(m) \right);\ m\leq t\leq M \right\}
\end{equation}}
{\small\begin{align} \label{eq:bf}
\left(resp. \ \b(m,M,f) = \min \left\{ \frac{1}{f(t)} \left( \frac{f(M)-f(m)}{M-m}(t-m)+f(m) \right);  m\leq t\leq M \right\}\right).
\end{align}}
\end{lemma}
\begin{theorem}\label{seo1}
Let $A_i$ be positive matrices with $m\leq A_i\leq M$ for some scalars $0<m<M$ \ $(i=1,\ldots,k)$. Let $\Phi$ be a unital positive multilinear map. If a real valued continuous function $f$ is super-multiplicative convex on $[m,M]$ (resp. sub-multiplicative concave), then
\[
\Phi(f(A_1),\ldots,f(A_k)) \leq \a(m^k,M^k,f) f(\Phi(A_1,\ldots,A_k))
\]
\[
({\rm resp.} \quad \b(m^k,M^k,f) f(\Phi(A_1,\ldots,A_k)) \leq \Phi(f(A_1),\ldots,f(A_k)),
\]
where $\a(m,M,f)$ and $\b(m,M,f)$ are defined by  \eqref{eq:af} and \eqref{eq:bf}, respectively.
\end{theorem}

\begin{proof}
We only prove the super-multiplicative convex case. Assume that $A_i=\sum_{j=1}^{q_i}\lambda_{ij}P_{ij}$ $(i=1,\cdots, k)$ is the spectral decomposition of each $A_i$ for which $\sum_{j=1}^{q_i}P_{ij}=I$. Then
{\small\begin{align}\label{qq3}
\Phi(f(A_1), f(A_2),\cdots, f(A_k))&=\Phi\left(\sum_{j=1}^{q_1}f(\lambda_{1j})P_{1j},\cdots,\sum_{j=1}^{q_k}f(\lambda_{kj})P_{kj}\right)\\
&=\sum_{j_1}^{q_1}\sum_{j_2}^{q_2}\cdots\sum_{j_k}^{q_k}f(\lambda_{1j_1})f(\lambda_{2j_2})\cdots f(\lambda_{kj_k})\Phi(P_{1j_1},\cdots, P_{kj_k}),\nonumber
\end{align}}
by multilinearity of $\Phi$. With  $C(j_1,\cdots, j_k):=\left(\Phi(P_{1j_1},\cdots, P_{kj_k})\right)^{\frac{1}{2}}$  we get
\begin{align*}
& \Phi(f(A_1),\cdots, f(A_k)) \\
& = \sum_{j_1}^{q_1}\sum_{j_2}^{q_2}\cdots\sum_{j_k}^{q_k} f(\lambda_{1j_1})f(\lambda_{2j_2})\cdots f(\lambda_{kj_k})\Phi(P_{1j_1},\cdots, P_{kj_k}) \\
& \leq \sum_{j_1}^{q_1}\sum_{j_2}^{q_2}\cdots\sum_{j_k}^{q_k}C(j_1,\cdots, j_k)f(\lambda_{1j_1}\lambda_{2j_2}\cdots \lambda_{kj_k})C(j_1,\cdots, j_k) \\
& \leq \a(m^k,M^k,f) f\left(\sum_{j_1}^{q_1}\sum_{j_2}^{q_2}\cdots\sum_{j_k}^{q_k}C(j_1,\cdots, j_k)\lambda_{1j_1}\lambda_{2j_2}\cdots \lambda_{kj_k}C(j_1,\cdots, j_k)\right)\\
& = \a(m^k,M^k,f) f\left(\sum_{j_1}^{q_1}\sum_{j_2}^{q_2}\cdots\sum_{j_k}^{q_k}\lambda_{1j_1}\lambda_{2j_2}\cdots \lambda_{kj_k}\Phi(P_{1j_1},\cdots, P_{kj_k})\right) \\
& = \a(m^k,M^k,f) f\left(\Phi(A_1,\cdots, A_k)\right).
\end{align*}
\end{proof}

\begin{corollary} \label{co-2-1}
Let $A_i$ \  $(i=1,\ldots,k)$ be positive matrices with $m\leq A_i\leq M$ for some scalars $0<m<M$.  Let $\Phi$ be a unital positive multilinear map. If  $r \not\in [0,1]$, then
\[
\Phi(A_1^r,\ldots,A_k^r) \leq K(m^k,M^k,r)\ \Phi(A_1,\ldots,A_k)^r.
\]
If $r\in [0,1]$, then the   inequality is reversed. Here the generalized Kantorovich constant $K(m,M,r)$ is defined by
\begin{equation} \label{eq:K}
K(m,M,r) = \frac{mM^r-Mm^r}{(r-1)(M-m)} \left( \frac{r-1}{r} \frac{M^r-m^r}{mM^r-Mm^r} \right)^r.
\end{equation}
In particular,
\[
\Phi(A_1^2,\ldots,A_k^2) \leq \frac{(M^k+m^k)^2}{4M^km^k} \Phi(A_1,\ldots,A_k)^2
\]
and
\[
\Phi(A_1^{-1},\ldots,A_k^{-1}) \leq \frac{(M^k+m^k)^2}{4M^km^k} \Phi(A_1,\ldots,A_k)^{-1}.
\]
\end{corollary}
\begin{theorem}\label{thmss}
Let $\sigma$ be a matrix mean with the representing function $f$ which is sub-multiplicative and not affine. Let $A_i,B_i$ be positive matrices with $m\leq A_i, B_i \leq M$ for some scalars $0<m<M$ and $i=1,\ldots,k$. Let $\Phi$ be a strictly positive unital multilinear map. Then
\[
\Phi(A_1 \sigma B_1,\ldots,A_k \sigma B_k) \geq \b\left(\frac{m^k}{M^k},\frac{M^k}{m^k},f\right) \Phi(A_1,\ldots,A_k)\ \sigma \ \Phi(B_1,\ldots,B_k)
\]
where $\b(m,M,f)$ is defined by \eqref{eq:bf}. In particular, for each $\a \in [0,1]$
\[
\Phi(A_1 \s_{\a} B_1,\ldots,A_k \s_{\a} B_k) \geq K\left(\frac{m^k}{M^k},\frac{M^k}{m^k},\a\right) \Phi(A_1,\ldots,A_k)\ \s_{\a} \ \Phi(B_1,\ldots,B_k)
\]
where $K(m,M,\a)$ is defined by \eqref{eq:K}.
\end{theorem}

\begin{proof}
Let $\Psi$ be a strictly positive multilinear map defined by
\[
\Psi(X_1,\ldots,X_k)=\Phi(A_1,\ldots,A_k)^{-\frac{1}{2}}\Phi(A_1^{\frac{1}{2}}X_1A_1^{\frac{1}{2}},\ldots,A_k^{\frac{1}{2}}X_kA_k^{\frac{1}{2}})\Phi(A_1,\ldots,A_k)^{-\frac{1}{2}}
\]
and $Y_i:=A^{\frac{-1}{2}}_iB_iA_i^{\frac{-1}{2}}$ $(i=1,\cdots k)$. Then it follows from $m/M \leq Y_i \leq M/m$ that
\begin{align*}
&\b\left(\frac{m^k}{M^k},\frac{M^k}{m^k},f\right)
 f\left(\Phi(A_1,\cdots, A_k)^{\frac{-1}{2}}\Phi(B_1,\cdots, B_k)\Phi(A_1,\cdots, A_k)^{\frac{-1}{2}}\right) \\
& = \b\left(\frac{m^k}{M^k},\frac{M^k}{m^k},f\right) f\big(\Psi(Y_1,\cdots, Y_k)\big) \\
& \leq \Psi\left(f(Y_1),\cdots, f(Y_k)\right)\quad (\mbox{by Theorem \ref{seo1}})\\
& =  \Phi(A_1,\cdots, A_k)^{\frac{-1}{2}}\Phi (A_1\sigma B_1,\cdots , A_k\sigma B_k)\Phi(A_1,\cdots, A_k)^{\frac{-1}{2}}.
\end{align*}
Multiplying both sides by $\Phi(A_1,\cdots, A_k)^{\frac{1}{2}}$ we get the desired result.
\end{proof}

  Utilizing  Lemma~\ref{lem-RCDJI}, we present  a reverse   additivity inequality for matrix means.

\begin{lemma} \label{lem-RAM}
Let $\sigma$ be a matrix mean with the representing function $f$, and $A,B,C,D$ positive matrices with $0<m\leq A,B,C,D \leq M$ for some scalars $m\leq M$. Then
\[
\b\left(\frac{m}{M},\frac{M}{m},f\right) \left[ (A+B)\ \sigma \ (C+D) \right] \leq A\ \sigma \ C + B \ \sigma \ D
\]
where $\b(m,M,f)$ is defined by \eqref{eq:bf}. In particular, for each $\a \in [0,1]$
\[
K(m/M,M/m,\a) \left[ (A+B)\ \s_{\a} \ (C+D) \right] \leq A\ \s_{\a} \ C + B \ \s_{\a} \ D
\]
and for $\a=1/2$
\[
\frac{2\sqrt[4]{Mm}}{\sqrt{M}+\sqrt{m}} \left[ (A+B)\ \s \ (C+D) \right] \leq A\ \s \ C + B \ \s \ D
\]
\end{lemma}

\begin{proof}
Put $X=A^{1/2}(A+B)^{-1/2}, Y=B^{1/2}(A+B)^{-1/2}, V=A^{-1/2}CA^{-1/2}$ and $W=B^{-1/2}DB^{-1/2}$. Since $m/M\leq V,W\leq M/m$ and $X^*X+Y^*Y=I$, it follows from Lemma~\ref{lem-RCDJI} that
\begin{align*}
& (A+B)\ \sigma \ (C+D) \\
& = (A+B)^{1/2}f(X^*VX+Y^*WY)(A+B)^{1/2} \\
& \leq \b(m/M,M/m,f)^{-1} (A+B)^{1/2}\left[ X^*f(V)X+Y^*f(W)Y \right] (A+B)^{1/2} \\
& = \b(m/M,M/m,f)^{-1}\left(  A\ \sigma \ C + B \ \sigma \ D \right).
\end{align*}
\end{proof}

We show a reverse inequality for positive multilinear maps and matrix means in Corollary~\ref{co-ML-OP}:
\begin{theorem}
Let $\sigma$ be a matrix mean with the representing function $f$, and $A,B,C,D$ positive matirces with $0<m\leq A,B,C,D \leq M$ for some scalars $m\leq M$. Let $\Phi$ be a unital positive multilinear map. Then
\begin{align*}
& \Phi(A,B)+\Phi(B,A) \\
& \leq \b\left( \frac{m^2}{M^2},\frac{M^2}{m^2},f \right)^{-2} \left[ \Phi(A\ \sigma \ B, A\ \sigma^{0} \ B)+\Phi(A\ \sigma^{0} \ B, A\ \sigma \ B) \right]
\end{align*}
where $\b(m,M,f)$ is defined by \eqref{eq:bf}.
\end{theorem}

\begin{proof}
Since $m^2\leq \Phi(A,B), \Phi(B,A)\leq M^2$, it follows that
\begin{align*}
& \Phi(A,B)+\Phi(B,A) \\
& = \left[ \Phi(A,B)+\Phi(B,A) \right] \ \sigma \ \left[ \Phi(B,A)+\Phi(A,B) \right] \\
& \leq \b\left( \frac{m^2}{M^2},\frac{M^2}{m^2},f \right)^{-1} \left[ \Phi(A,B)\ \sigma \ \Phi(B,A) + \Phi(B,A)\ \sigma \ \Phi(A,B) \right] \ \ (\mbox{by Lemma \ref{lem-RAM}})\\
& \leq \b\left( \frac{m^2}{M^2},\frac{M^2}{m^2},f \right)^{-2} \left[ \Phi(A\ \sigma \ B, A\ \sigma^{0} \ B)+\Phi(A\ \sigma^{0} \ B, A\ \sigma \ B) \right],
\end{align*}
where the last inequality follows from Theorem \ref{thmss}.
\end{proof}

\section{Information monotonicity for Karcher mean}

In Proposition~\ref{thm-s}, we showed the positive multilinear map version of Ando's inequality \eqref{eq:Ando} due to matrix means. In this section, we consider its $n$-variable positive multilinear map version. For this, we recall the notion of the Karcher mean and matrix power mean: Let ${\Bbb A}=(A_1,\ldots,A_k)$ be a $k$-tuple of strictly positive matrices and $\omega=(\omega_1,\ldots,\omega_k)$ a weight vector such that $\omega_i\geq 0$ and $\sum_{i=1}^k \omega_i =1$. The Karcher mean of $A_1,\ldots,A_k$ is the unique positive invertible solution of the Karcher equation
\[
\sum_{i=1}^k \omega_i \log (X^{-1/2}A_iX^{-1/2}) = 0
\]
 and we denote it by $G_{\rm K}(\omega;{\Bbb A})=G(\omega;{\Bbb A})$. The matrix power mean of $A_1,\ldots,A_k$ is the unique positive invertible solution of a non-linear matrix equation
\[
X=\sum_{i=1}^k \omega_i (X\ \sharp_{t}\ A_i) \qquad \mbox{for $t\in (0,1]$}
\]
and we denote it by $P_t(\omega;{\Bbb A})$. For $t\in [-1,0)$, we define $P_t(\omega;{\Bbb A}) := P_{-t}(\omega;{\Bbb A}^{-1})^{-1}$, where ${\Bbb A}^{-1}=(A_1^{-1},\ldots,A_k^{-1})$. The matrix power mean satisfies all desirable properties of power arithmetic means of positive real numbers and interpolates between the weighted harmonic and arithmetic means. Moreover, the Karcher mean coincides with the limit of matrix power means as $t\to 0$. The matrix power mean satisfies an information monotonicity: For each $t\in (0,1]$
\begin{equation} \label{eq:PI}
\Phi(P_t(\omega;{\Bbb A})) \leq P_t(\omega; \Phi({\Bbb A}))
\end{equation}
for any unital positive linear map $\Phi$. For more details on the Karcher mean and the matrix power mean; see \cite{Bh,BH,LL2014,LP}.\par
 In the final section, we show a positive multilinear map version of \eqref{eq:PI}. Let $\Phi$ be a unital positive multilinear map and $\omega^{(i)}=(\omega_1^{(i)},\ldots,\omega_k^{(i)}$ weight vectors and ${\Bbb A}^{(i)}=(A_1^{(i)},\ldots,A_k^{(i)})$ $k$-tuples of strictly positive matrices for $i=1,\ldots,n$. Define
\begin{align*}
\Phi({\Bbb A}^{(1)},\ldots, {\Bbb A}^{(n)}) & := (\Phi(A_1^{(1)},A_1^{(2)},\ldots,A_1^{(n)}), \Phi(A_2^{(1)},A_2^{(2)},\ldots,A_2^{(n)}), \ldots, \\
& \Phi(A_k^{(1)},A_k^{(2)},\ldots,A_k^{(n)}))
\end{align*}
and
\[
\omega^{(1)}\cdots \omega^{(n)} := (\omega_1^{(1)}\omega_1^{(2)}\cdots \omega_1^{(n)}, \omega_2^{(1)}\omega_2^{(2)}\cdots \omega_2^{(n)}, \ldots, \omega_k^{(1)}\omega_k^{(2)}\cdots \omega_k^{(n)}).
\]

\begin{theorem} \label{thm-IMP}
Let $\Phi$ be a unital positive multilinear map. If $t\in (0,1]$, then
\[
\Phi( P_t(\omega^{(1)};{\Bbb A}^{(1)}), \ldots, P_t(\omega^{(n)};{\Bbb A}^{(n)})) \leq P_t(\omega^{(1)}\omega^{(2)}\cdots \omega^{(n)};\Phi({\Bbb A}^{(1)},{\Bbb A}^{(2)},\ldots, {\Bbb A}^{(n)})).
\]
If $t\in [-1,0)$ and $0<m\leq A_j^{(i)}\leq M$ for $j=1,\ldots,k$ and $i=1,\ldots,n$ and some scalars $0<m<M$, then
\begin{align*}
P_t(\omega^{(1)}\omega^{(2)}\cdots & \omega^{(n)};\Phi({\Bbb A}^{(1)},{\Bbb A}^{(2)},\ldots, {\Bbb A}^{(n)}) ) \\
& \leq \frac{(M^n+m^n)^2}{4M^nm^n} \Phi( P_t(\omega^{(1)};{\Bbb A}^{(1)}), \ldots, P_t(\omega^{(n)};{\Bbb A}^{(n)})).
\end{align*}
\end{theorem}

\begin{proof}
We only prove the case of $n=2$: Let ${\Bbb A}=(A_1,\ldots,A_k), {\Bbb B}=(B_1,\ldots,B_k)$ be $k$-tuples of positive matrices and $\omega=(\omega_1,\ldots,\omega_k), \omega'=(\omega_1',\ldots,\omega_k')$ weight vectors. Put $X_t=P_t(\omega;{\Bbb A})$ and $Y_t=P_t(\omega';{\Bbb B})$. Then we have $X_t=\sum \omega_i (X_t\ \s_{t} \ A_i$ and $Y_t=\sum \omega_i' (Y_t \ \s_{t}\ B_i)$. By Proposition~\ref{thm-s}, it follows that
\begin{align}\label{q11}
\Phi(X_t,Y_t) & = \sum_{ij=1}^k \omega_i \omega_j' \Phi(X_t\ \s_{t}\ A_i, Y_t\ \s_{t}\ B_j) \nonumber\\
& \leq \sum_{ij=1}^k \omega_i\omega_j' \Phi(X_t,Y_t)\ \s_{t}\ \Phi(A_i,B_j).
\end{align}
Define a map $f(X)=\sum_{ij=1}^k  \omega_i \omega_j' X\ \s_{t}\ \Phi(A_i,B_j)$ and then it follows from \cite[Theorem 3.1]{LP} that $f^n(X) \to P_t(\omega^{(2)};\Phi({\Bbb A},{\Bbb B}))$ for all $X>0$ as $n\to \infty$. On the other hand, $\Phi(X_t,Y_t)\leq f(\Phi(X_t,Y_t))$ by \eqref{q11}. This follows that $\Phi(X_t,Y_t) \leq f^n(\Phi(X_t,Y_t))$ for any $n$. Hence as $n\to \infty$ we have the desired inequality for $t\in (0,1]$. \par

Let $t\in [-1,0)$. By Corollary~\ref{co-2-1}, $\Phi({\Bbb A}^{-1}) \leq \frac{(M^k+m^k)^2}{4M^km^k} \Phi({\Bbb A})^{-1}$ for all $k$-tuple of strictly positive matrices. Therefore,
\begin{align*}
\Phi(P_{-t}(\omega;{\Bbb A}^{-1}), P_{-t}(\omega'; {\Bbb B}^{-1})) & \leq P_{-t}(\omega \omega';\Phi({\Bbb A}^{-1},{\Bbb B}^{-1})) \\
& \leq P_{-t}(\omega \omega';\frac{(M^2+m^2)^2}{4M^2m^2}\Phi({\Bbb A},{\Bbb B})^{-1}) \\
& = \frac{(M^2+m^2)^2}{4M^2m^2} P_{-t}(\omega \omega';\Phi({\Bbb A},{\Bbb B})^{-1})
\end{align*}
and this implies
\begin{align*}
P_t(\omega \omega'; \Phi({\Bbb A},{\Bbb B})) & = P_{-t}(\omega \omega'; \Phi({\Bbb A},{\Bbb B})^{-1})^{-1} \\
& \leq \frac{(M^2+m^2)^2}{4M^2m^2} \Phi(P_{-t}(\omega; {\Bbb A}^{-1}), P_{-t}(\omega';{\Bbb B}^{-1}))^{-1} \\
& \leq \frac{(M^2+m^2)^2}{4M^2m^2} \Phi(P_{-t}(\omega; {\Bbb A}^{-1})^{-1}, P_{-t}(\omega';{\Bbb B}^{-1}) \\
& = \frac{(M^2+m^2)^2}{4M^2m^2} \Phi(P_{t}(\omega; {\Bbb A}), P_{t}(\omega';{\Bbb B})).
\end{align*}
\end{proof}

\begin{theorem}
If $\Phi$ is unital positive multilinear map, then
\begin{align*}
\Phi( G(\omega^{(1)};{\Bbb A}^{(1)}), & \ldots, G(\omega^{(n)};{\Bbb A}^{(n)})) \leq G(\omega^{(1)}\cdots \omega^{(n)};\Phi({\Bbb A}^{(1)}, \ldots, {\Bbb A}^{(n)})) \\
& \leq \frac{(M^n+m^n)^2}{4M^nm^n} \Phi( G(\omega^{(1)};{\Bbb A}^{(1)}), \ldots, G(\omega^{(n)};{\Bbb A}^{(n)})).
\end{align*}
\end{theorem}

\begin{proof}
By Theorem~\ref{thm-IMP},
\[
\Phi( P_t(\omega^{(1)};{\Bbb A}^{(1)}), \ldots, P_t(\omega^{(n)};{\Bbb A}^{(n)})) \leq P_t(\omega^{(1)}\cdots \omega^{(n)};\Phi({\Bbb A}^{(1)}, \ldots, {\Bbb A}^{(n)}))
\]
for all $t\in (0,1]$. As $t\to 0$, we have
\[
\Phi( G(\omega^{(1)};{\Bbb A}^{(1)}), \ldots, G(\omega^{(n)};{\Bbb A}^{(n)})) \leq G(\omega^{(1)}\cdots \omega^{(n)};\Phi({\Bbb A}^{(1)}, \ldots, {\Bbb A}^{(n)})).
\]
Also, since
\[
\frac{(M^n+m^n)^2}{4M^nm^n}\Phi( P_t(\omega^{(1)};{\Bbb A}^{(1)}), \ldots, P_t(\omega^{(n)};{\Bbb A}^{(n)})) \geq P_t(\omega^{(1)}\cdots \omega^{(n)};\Phi({\Bbb A}^{(1)}, \ldots, {\Bbb A}^{(n)}))
\]
for all $t\in [-1,0)$, as $t\to 0$, we have
\[
\frac{(M^n+m^n)^2}{4M^nm^n}\Phi( G(\omega^{(1)};{\Bbb A}^{(1)}), \ldots, G(\omega^{(n)};{\Bbb A}^{(n)})) \geq G(\omega^{(1)}\cdots \omega^{(n)};\Phi({\Bbb A}^{(1)}, \ldots, {\Bbb A}^{(n)})).
\]
\end{proof}

\begin{remark}
Let $\Phi$ be a unital positive linear map. Then for each $t\in [-1,0)$
\[
P_t(\omega;\Phi({\Bbb A})) \leq \frac{(M+m)^2}{4Mm} \Phi( P_t(\omega;{\Bbb A}))
\]
and this is a complementary result of \cite[Proposition 3.6 (13)]{LL2014}.
\end{remark}


\end{document}